\documentclass{article}
\usepackage{amsmath,amssymb,amsfonts}
\usepackage{graphicx}

\def\deg {\mathrm{deg}}

\def\Jac{\mathrm{Jac}}

\newtheorem{theorem}{Theorem}
\newtheorem{proposition}{Proposition}
\newtheorem{corollary}{Corollary}
\newtheorem{lemma}{Lemma}

\newtheorem{definition}{Definition}
\newenvironment{proof}[1][Proof]{\noindent\textit{#1.} }{\hfill$\Box$\medskip}

\title{Multi-valued hyperelliptic continued fractions of generalized Halphen type}

\author{Vladimir Dragovi\'c}

\date{}

\begin{document}

\maketitle

\medskip

\centerline{Mathematical Institute SANU}

\centerline{Kneza Mihaila 35, 11000 Belgrade, Serbia}

\smallskip

\centerline{Mathematical Physics Group, University of Lisbon}

\smallskip

\centerline{e-mail: {\tt vladad@mi.sanu.ac.yu}}

\

\

\begin{abstract}
\smallskip
We introduce and study higher genera generalizations of the
Halphen theory of continued fractions. The basic notion we start
with is hyperelliptic Halphen (HH) element
$$\frac{\sqrt{X_{2g+2}}-\sqrt{Y_{2g+2}}}{x-y},$$
depending on parameter $y$, where $X_{2g+2}$ is a polynomial of
degree $2g+2$ and $Y_{2g+2}=X_{2g+2}(y)$. We study regular and
irregular HH elements, their continued fraction development and
some basic properties of such developments such as: even and odd
symmetry and periodicity.
\end{abstract}

\

\

\newpage

\tableofcontents

\newpage

\section{Introduction}

\smallskip

Modern  algebraic approximation theory with continued fraction
theory was established by Tchebycheff and his Sankt Petrsburg
school in the second half of the XIX century. Tchebycheff
motivation for this studies was his interest in practical
problems: in the mechanism theory as an important part of
mechanical engineering of that time and ballistics. Steam engines
were fundamental tool in technological revolution and their kernel
part was {\it the Watt's complete parallelogram}, a planar
mechanism to transform linear motion into circular.

\begin{figure}[h]
\centering
\begin{minipage}[b]{0.6\textwidth}
\centering
\includegraphics[width=8cm,height=7cm]{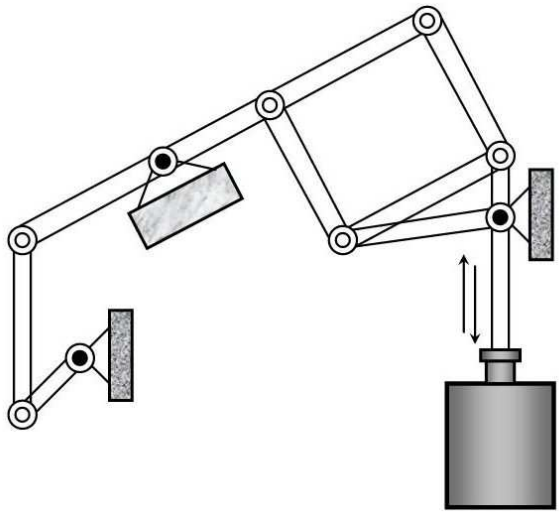}
\caption{Watt's complete parallelogram}
\end{minipage}\label{fig:masina}
\end{figure}

Fundamental problem was to estimate error of the mechanism in
execution of that transformation.

Starting point of Tchebycheff investigation (\cite{Tcheb1}) was
work on the theory of mechanisms of French military engineer,
professor of mechanics and academician Jean Victor Poncelet
\cite{Pon1}. In his study of mistakes of mechanisms, Poncelet came
to the question of rational and linear approximation of the
function
$$
\sqrt{x^2+1}.
$$
In other words he studied approximation of the functions
$\sqrt{X_2(x)}$ of the form of square root of polynomials of the
second degree, and he gave two approaches to the posed problems,
one based on the analytical arguments and the second one based on
geometric consideration.

Although Poncelet was described by Tchebycheff as "well-known
scientist in practical mechanics" (see \cite{Tcheb}), nowadays J.
V. Poncelet is known first of all as one of the biggest geometers
of the XIX century. The Great Poncelet Theorem (GPT) is considered
as one of the nicest and deepest results in projective geometry:
Suppose that two ellipses are given in the plane, together with a
closed polygonal line inscribed in one of them and circumscribed
about the other one. Then, GPT states that infinitely many such
closed polygonal lines exist –-- every point of the first ellipse
is a vertex of such a polygon. Besides, all these polygons have
the same number of sides.

\begin{figure}[h]
\centering
\begin{minipage}[b]{0.6\textwidth}
\centering
\includegraphics[width=7cm,height=4.5cm]{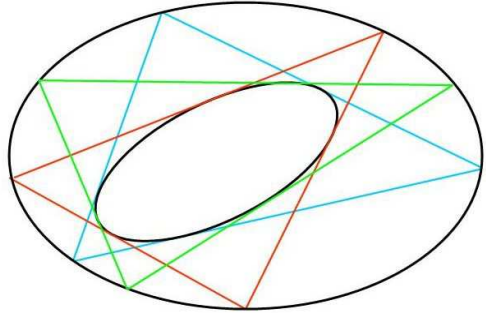}
\caption{Poncelet theorem}
\end{minipage}\label{fig:GPT}
\end{figure}

In his {\it Trait\'{e} des propri\'et\'es projectives des figures}
\cite{Pon}, Poncelet proved even more general result and used only
purely geometric, synthetic arguments.

(The case when the two ellipses are confocal has clear mechanical
interpretation as billiard system within outer ellipse as boundary
and having inner ellipse as the caustic of a given billiard
trajectory. GPT in this case describes periodic billiard
trajectories. For a modern account of GPT see \cite{DR,DR2} and
references therein.)

\begin{figure}[h]
\centering
\begin{minipage}[b]{0.6\textwidth}
\centering
\includegraphics[width=8cm,height=6cm]{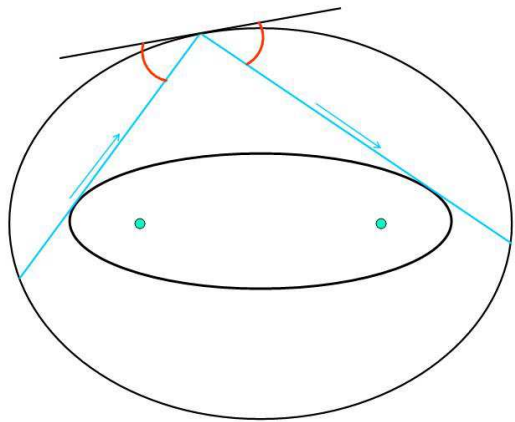}
\caption{billiard system and confocal conics}
\end{minipage}\label{fig:confocal}
\end{figure}

However, nowadays it is almost forgotten that there  also exists
amazing connection between Great Poncelet Theorem and continued
fractions and approximation theory of the functions of the form
$$
\sqrt{X_4(x)},
$$
where $X_4(x)$ denotes general polynomial of the fourth degree.
This connection of continued fractions and approximations of
functions of the form of square root of polynomials of fourth
degree with the Poncelet configuration and GPT was indicated by
Halphen \cite{Ha}.

\begin{figure}[h]
\centering
\begin{minipage}[b]{0.6\textwidth}
\centering
\includegraphics[width=4.8cm,height=5cm]{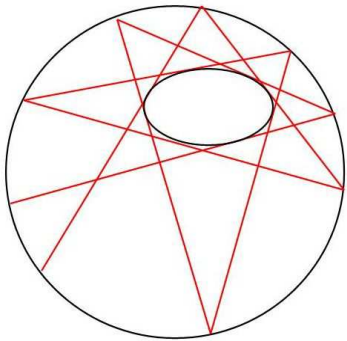}
\caption{the Poncelet configuration}
\end{minipage}\label{fig:Ponceletconf}
\end{figure}

 Theory of
continued fractions of square roots of polynomials of degree up to
four $\sqrt {X_4(x)}$ started with Abel and Jacobi. Meeting
problems with very complicated algebraic formulae in algorithm,
Jacobi \cite {Ja1} turned to an approach based on elliptic
function theory. Further development of that approach has been
done by Halphen \cite{Ha}. Halphen studied, instead of square root
of a polynomial, the following, more general, {\it Halphen
element}
\begin{equation}\label{eq:halphen1}\frac{\sqrt{X_4}-\sqrt{Y_4}}{x-y},
\end{equation}
where $Y_4=X_4(y)$ is the value of the polynomial $X$ at a given
point $y$.

In this paper we are going to study more general theory of
continued fractions of {\it hyperelliptic Haplhen elements}
\begin{equation}\label{eq:halphen2}\frac{\sqrt{X_{2g+2}}-\sqrt{Y_{2g+2}}}{x-y},
\end{equation}
where $X_{2g+2}$ is a polynomial of degree $2g+2$ and
$Y_{2g+2}=X_{2g+2}(y)$. It is obviously related to the theory of
functions of the hyperelliptic curve
$$\Gamma: z^2=X_{2g+2}$$
of genus $g$. We are also going to refer to this theory as {\it HH
continued fractions}. We hope that this theory will quickly find
its way to concrete applications in modern technology. As a
possibility we can mention development of {\it branched,
multivalued algorithms} to be used in future cryptology.

\section{Basic Algebraic Lemma}\label{sekcija:fund}

\smallskip

Given a polynomial $X$ of degree $2g+2$ in $x$. We suppose that
$X$ is not a square of a polynomial. Assuming that the values of
$y$ and $\epsilon$ are finite and fixed, we are going to study
$HH$ elements in a neighborhood of $\epsilon$. Then, $X$ can be
considered as a polynomial of degree $2g+2$ in $s$, where
$s=x-\epsilon$ is chosen  as a variable in a neighborhood of
$\epsilon$.
\medskip

\begin{lemma}\label{th:fund}[Basic Algebraic Lemma] Let $X$ be a polynomial of
degree $2g+2$ in $x$ and $Y=X(y)$ its value at a given fixed point
$y$. Then, there exists a unique triplet of polynomials $A, B, C$
with $\deg A=g+1$, $\deg B=\deg C= g$ in $x$ such that
\begin{equation}\label{eq:halphen3}\frac{\sqrt{X}-\sqrt{Y}}{x-y}-
C=\frac{B(x-\epsilon)^{g+1}}{\sqrt{X} + A}.
\end{equation}
\end{lemma}
\begin{proof} Put $s=x-\epsilon$ and $t=y-\epsilon$ and denote
$X=X'(s)=\sum_{i=0}^{2g+2}p_is^i$ and
$$A=\sum_{i=0}^{g+1}A_is^i,\quad B=\sum_{i=0}^gB_is^i,\quad
C=\sum_{i=0}^gC_is^i.
$$
We are going to determine the coefficients $A_i, B_i, C_i$ in the
way that the equation (\ref{eq:halphen3}) is satisfied. The last
equation can be separated into two equations taking into account
that $\sqrt{X}$ is irrational:
\begin{equation}\label{eq:halpsep}
\aligned A-\sqrt{Y}-C(s-t)=0;\\
X-A\sqrt{Y}-AC(s-t)-Bs^{g+1}(s-t)=0.
\endaligned
\end{equation}
We also add the next equation which is a consequence of the last
two:
\begin{equation}\label{eq:halpadd}
X-A^2=Bs^{g+1}(s-t).
\end{equation}
We obtain from equation (\ref{eq:halphen3}) for $s=0$
$$
C_0=\frac{\sqrt{Y}-\sqrt{p_0}}{t}
$$
and again for $s=0$ from equation (\ref{eq:halpsep})
$$
A_0=-C_0t+\sqrt{Y}=\sqrt{p_0}.
$$
Then we calculate $A_i$, $i=1,\dots, g$ from equation
(\ref{eq:halpadd}). From the first of equations
(\ref{eq:halpsep}), by putting $s=t$, we get
$$
A(t)=\sqrt{Y}.
$$
From the first equation (\ref{eq:halpsep}) for $s=t$ we see that
$\deg C=\deg A-1$ and we compute all the coefficients $C_i$ as
functions of the coefficients of the polynomial $A$. For example,
$C_g=A_{g+1}$.

The last step is to compute the polynomial $B$. Observe that the
coefficients $A_0,\dots A_g$ of the polynomial $A$ are obtained in
a manner such that
$$s^{g+1}|X-A^2.$$
The leading coefficient $A_{g+1}$ is such that $X-A^2=0$ for
$s=t$. Thus, there is a unique polynomial $B$, $\deg B=g$ such
that
$$
X-A^2=Bs^{g+1}(s-t).
$$
\end{proof}

\section{Hyperelliptic Halphen-type continued fractions (HH c.f.)}
\label{sec:hhcf}

\smallskip

Let us start with the factorization of the polynomial $B$:
$$
B(s)=B_g\prod_{i=1}^g(s-t_1^i)
$$
and denote $A(t_1^i)=-\sqrt{Y_1^i}$. Then we have
$$
\frac{A+\sqrt{X}}{s-t_1^i}=P_A^g(t_1^i,s)+\frac{\sqrt{X}-\sqrt{Y_1^i}}{x-y_1^i},
$$
with certain polynomial $P_A^g$ of degree $g$ in $s$ and with
coefficients depending on the coefficients of $A$ and $t_1^i$.

Denote
$$Q_0=\frac{\sqrt{X}-\sqrt{Y}}{x-y}-C.
$$
Then we have
$$
Q_0=\frac{B_g\prod_{j=1, j\neq
i}^g(s-t_1^j)s^{g+1}}{P_A^g(t_1^i,s)+\frac{\sqrt{X}-\sqrt{Y_1^i}}{x-y_1^i}}.
$$
Now, by applying the Lemma \ref{th:fund} we obtain the polynomials
$A^{(1,i)}, B^{(1,i)}, C^{(1,i)}$  of degree $g+1, g, g$
respectively, such that
$$\frac{\sqrt{X}-\sqrt{Y_1^i}}{x-y_1^i}-
C^{(1,i)}=\frac{B^{(1,i)}(x-\epsilon)^{g+1}}{\sqrt{X} +
A^{(1,i)}}.
$$
Denote
$$
\alpha_1^{(i)}:=P_A^g(t_1^i,s),\quad \beta_1^{(i)}:=B_g\prod_{j=1,
j\neq i}^g(s-t_1^j)s^{g+1},
$$
and introduce $Q_1^{(i)}$ by the equation
$$
Q_{0}=\frac{\beta_1^{(i)}}{\alpha_1^{(i)} + Q_1^{(i)}}.
$$
Observe that $\deg \alpha_1^{(i)} =g$ and $\deg \beta_1^{(i)}=2g$.

Now, one can go further, step by step: to factorize $B^{(1,i)}$,
to choose one of its zeroes $t_2^j$ and to denote by
$B^{i,j}:=B^{(1,i)}/(s-t_2^j)$. Further, we denote
$$
\alpha_2^{(i,j)}:=P_{A^{1,i}}g(t_2^j,s),\quad
\beta_2^{(i,j)}:=B^{i,j}s^{g+1},
$$
and calculate $Q_2^{(i,j)}$ from the equation
$$
Q_1^{(i)}=\frac{\beta_2^{(i,j)}}{\alpha_2^{(i,j)} + Q_2^{(i,j)}}.
$$
Thus we have
$$
\frac{\sqrt{X}-\sqrt{Y}}{x-y}=C+\frac{\beta_1^{(i)}}{\alpha_1^{i}+\frac{\beta_2^{(i,j)}}
{\alpha_2^{(i,j)}+Q_2^{(i,j)}}}.
$$
Following the same scheme, in the $i$-th step we introduce
polynomials
$$
A^{(i,j_1,\dots,j_i)},\quad B^{(i,j_1,\dots,j_i)},\quad
C^{(i,j_1,\dots,j_i)}
$$
of degrees $g+1,\quad g,\quad g$ respectively. They satisfy the
equations
\begin{equation}\label{eq:halphi}
\aligned
A^{(i,j_1,\dots,j_i)}&=C^{(i,j_1,\dots,j_i)}(s-t_i^{j_1,\dots,j_i})+\sqrt{Y_i^{j_1,\dots,j_i}},\\
X-A^{(i,j_1,\dots,j_i)2}&=B^{(i,j_1,\dots,j_i)}s^{g+1}(s-t_i^{j_1,\dots,j_i}).
\endaligned
\end{equation}
We see that in the case $g>1$ the formulae of the $i+1$-th step
depend on the choice of one of the roots of the polynomial
$B^{(i)}$ and of the choices from the previous steps. To avoid
abuse of notations we are going to omit many times in future
formulae the indexes $j_1,\dots,j_i$, which indicate the choices
done in the first $i$ steps, although we assume all the time that
the choice has been done.

According to our notation we have
$$
s-t_i|B^{(i-1)}
$$
and
$$B^{(i-1)}=\frac{\beta_i}{s^{g+1}}(s-t_i)$$
or
$$
B^{(i)}=\hat\beta_{i+1}(s-t_{i+1}),
$$
where $\hat\beta_{i}=\beta_i/s^{g+1}$. From the equations
(\ref{eq:halphi}) we have
\begin{equation}\label{eq:sistem}
\aligned X-A^{(i-1)2}&=\hat\beta_{i}(s-t_{i-1})s^{g+1}(s-t_i),\\
X-A^{(i)2}&=\hat\beta_{i+1}(s-t_{i+1})s^{g+1}(s-t_i)
\endaligned
\end{equation}
together with
$$
\aligned A^{(i)}(t_i)&=\sqrt{Y_i},\\
 A^{(i-1)}(t_i)&=-\sqrt{Y_i}.
\endaligned
$$
We introduce $\lambda_i$ by the relation
$$
A_{g+1}^{(i)}=\sqrt{p_0}\lambda_i.
$$
\begin{theorem} \label{th:rel1} If  $\lambda_i$ is  fixed, then $t_i$ and
$\{t_{i+1}^{(1)},\dots, t_{i+1}^{(g)}\}$ are the roots of
polynomial equation of degree $g+1$ in $s$
$$Q_X(\lambda_i,s)=0.$$
\end{theorem}

The proof follows from the equations (\ref{eq:sistem}). On the
same way we get
\begin{theorem}\label{th:rel2} If  $t_i$ is  fixed, then $\lambda_i$ and
$\lambda_{i-1}$ are the roots of the polynomial equation of degree
2 in $\lambda$:
$$Q_X(\lambda_{i-1}, t_i)=0,\qquad Q_X(\lambda_{i}, t_i)=0.$$
\end{theorem}
One can easily calculate
$$ B_g^{(i)}=p_{2g+2}-p_0\lambda_i^2,$$
thus
$$\beta_{i+1}=(p_{2g+2}-p_0\lambda_i^2)\prod_{j=2}^g(s-t_i^j)s^{g+1}.$$
We also have
$$
\aligned A^{(i)}&=\sqrt{p_0}\left(1+q_1s+\dots+\lambda_is^{g+1}\right),\\
C^{(i)}&=\sqrt{p_0}\left(q_1+\dots+\lambda_i(s^{g}+s^{g-1}t_i+\dots+t_i^g)\right),
\endaligned
$$
and
$$
\alpha_i=\sqrt{p_0}\left(2q_1+\dots+(\lambda_{i-1}+\lambda_i)(s^{g}+s^{g-1}t_i+\dots+t_i^g)\right).
$$
In the last equation the sum $\lambda_{i-1}+\lambda_i$ can be
expressed through the coefficients of the polynomial $Q_X(\lambda,
t_i)$ as a polynomial of the second degree in $\lambda$ according
to the Theorem \ref{th:rel2}.
\medskip

\section{Basic examples: genus one case}\label{sec:g1}

\smallskip

The genus one case, or the elliptic case, has been studied by
Halphen. Here we reproduce some of his formulae. See \cite{Ha} for
more details. The elliptic curve is given by  a polynomial $X$ of
degree $4$, in variable $s$ in a neighborhood of $\epsilon$
$$
X=S(s)=\sum_{i=0}^{4}p_is^i
$$
The development around the point $\epsilon$ of its  square root
has the form
$$
\sqrt{X}=\sqrt{p_0}\left(1+q_1s +q_2s^2+q_3s^3+\dots \right),
$$
with relations between $q$'s and $p$'s:
$$
\aligned q_1&=\frac{p_1}{2p_0},\\
q_2&=\frac{1}{8p_0^2}\left(4p_0p_2-p_1^2\right),\\
q_3&=\frac{1}{4p_0^3}\left(2p_0p_3-p_0p_1p_2+\frac{p_1^3}{4}\right),\\
q_4&=\frac{1}{2p_0}\left(p_4-2q_1q_3p_0-q_2^2p_0\right).
\endaligned
$$
Here we have
$$
\frac {X}{p_0}=(1+q_1s+q_2s^2)^2+2q_3s^3+2(q_1q_3+q_4)s^4.
$$
 From the Basic Algebraic Lemma, applied in the case $g=1$,
we get the polynomials $A=A_0+A_1s+A_2s^2$, $B=B_0+B_1s,
C=C_0+C_1s$ which satisfy
\begin{equation}\label{eq:halpsep1}
\aligned A-\sqrt{Y}-C(s-t)=0;\\
X-A\sqrt{Y}-AC(s-t)-Bs^{2}(s-t)=0;\\
 X-A^2=Bs^{2}(s-t).
 \endaligned
\end{equation}
From the equations (\ref{eq:halpsep1}) one gets the formulae for
the polynomials $A, B, C$:
$$A_0=\sqrt{p_0},\quad A_1=q_1\sqrt{p_0},\quad
A_2=\frac{\sqrt{Y}-(1+q_1t)\sqrt{p_0}}{t^2},$$
$$
C_0=\frac{\sqrt{Y}-\sqrt{p_0}}{t},\quad C_1=A_2,
$$
$$
\aligned
B_0&=\frac{2\sqrt{p_0}}{t^3}\left(\sqrt{Y}-\sqrt{p_0}(1+q_1t+q_2t^2)\right),\\
B_1&=\frac{2\sqrt{p_0}}{t^4}\left((1+q_1t)\sqrt{Y}-\sqrt{p_0}(1+2q_1t+(q_1^2+q_2)t^2+(q_1q_2+q_3)t^3)\right).
\endaligned
$$
If we denote
$$
P_A^{(1)}(t,s):=A_1+A_2(s+t),
$$
then we have
$$Q_0=\frac{B_1s^2}{P_A^{(1)}(t_1,s)+\frac{\sqrt{X}-\sqrt{Y_1}}
{x-y_1}},$$
 and, step by step
\begin{equation}\label{eq:halpsep1i}
\aligned A^{(i)}=\sqrt{Y_i}-C^{(i)}(s-t_i);\\
 X-A^{(i)2}=B^{(i)}s^{2}(s-t_i),
 \endaligned
\end{equation}
where
$$\aligned
B^{(i-1)}(t_i)&=0,\\
A^{(i)}(t_{i+1})&=-\sqrt{Y_{i+1}}.
\endaligned
$$
Finally, one gets
\begin{equation}\label{eq:alpha1}
\aligned \beta_i&=B_1^{(i-1)}s^2,\\
\alpha_i&=P_{A^{(i-1)}}^{(1)}(t_i, s)+ C^{(i)}.
\endaligned
\end{equation}

From the equation (\ref{eq:halpsep1i}) we get
\begin{equation}\label{eq:halpsep11i}
\aligned
 X-A^{(i-1)2}&=ms^{2}(s-t_{i-1})(s-t_i);\\
X-A^{(i)2}&=ns^{2}(s-t_i)(s-t_{i+1});\\
 A^{(i)}(t_{i})&=\sqrt{Y_{i}};\\
 A^{(i-1)}(t_{i})&=-\sqrt{Y_{i}};\\
 A_2^{(i)}&=\sqrt{p_0}\lambda_i,
\endaligned
\end{equation}
with some constants $m, n$, and then we get
\begin{equation}
\aligned
\lambda_i&=\frac{1}{t_i^2}\left(\frac{\sqrt{Y_i}}{\sqrt{p_0}}-(1+q_1t_i)\right),\\
\lambda_{i-1}&=\frac{1}{t_i^2}\left(-\frac{\sqrt{Y_i}}{\sqrt{p_0}}-(1+q_1t_i)\right),
\endaligned
\end{equation}
From the last equations one gets
\medskip
\begin{proposition} If $\lambda_i$ is fixed, then $t_i, t_{i+1}$
are roots of the polynomial $Q_X(\lambda_i,s)$ quadratic in $s$:
$$
Q_X(\lambda_i,s):=(p_4-p_0\lambda_i^2)s^2+(p_3-p_1\lambda_i)s+
2p_0(q_2-\lambda_i)=0.
$$
\end{proposition}

\medskip

\begin{corollary} Product of two consecutive $t_i$ and $t_{i+1}$ is
$$
t_it_{i+1}=\frac{2p_0(\lambda_i-q_2)}{p_0\lambda_i^2-p_4},
$$
and the sum is
$$
t_i+t_{i+1}=\frac{p_1\lambda_i-p_3}{p_4-p_0\lambda_i^2}.
$$
\end{corollary}
\medskip
The last Proposition can be reformulated giving relation between
two consecutive $\lambda_{i-1}$ and $\lambda_i$:
\medskip
\begin{proposition}
If $t_i$ is fixed, then $\lambda_{i-1},
\lambda_i$ are solutions of quadratic equation:
$$
\lambda^2(p_0t_i^2)+\lambda(p_1t_i+2p_0) -
(p_4t_i^2+p_3t_i+2p_0q_2)=0.
$$
\end{proposition}
\medskip
For the {\it normal form} of the elliptic HH c. f. we consider the
case where
$$
\alpha_i'=1+u_is,\qquad \beta_i'=v_is^2.
$$
Then we have
$$t_i=-\frac{1}{q_1+u_i},\qquad \lambda_i=q_2-2v_{i+1}.$$
The recurrent relations are given with
\begin{equation}\label{rec11}
\aligned u_i+u_{i-1}&=-q_1+\frac{q_2}{2v_i},\\
v_i+u_iu_{i-1}&=q_2+\frac{q_4}{2v_i}.
\endaligned
\end{equation}
The second set of recurrent equations is done by
\begin{equation}\label{rec12}
\aligned v_i+v_{i+1}&=q_2+q_1u_i+u_i^2,\\
2v_iv_{i+1}&=-q_4+q_3u_i.
\endaligned
\end{equation}

\medskip
\section{Basic examples: genus two case}\label{sec:g2}
\subsection{Notation}
\smallskip

We start with a polynomial $X$ of degree $6$ in $x$ and rewrite it
as a polynomial in $s$ in a neighborhood of $\epsilon$
$$
X=S(s)=\sum_{i=0}^{6}p_is^i
$$
and its  square root  developed around $\epsilon$ as
$$
\sqrt{X}=\sqrt{p_0}(1+q_1s
+q_2s^2+q_3s^3+q_4s^4+q_5s^5+q_6s^6+q_7s^7+\dots) .
$$
Then, the relations between coefficients $p_i$ and $q_j$ are
$$
\aligned p_1&=2p_0q_1,\\
p_2&=p_0(2q_2+q_1^2),\\
p_3&=2p_0(q_3+q_1q_2),\\
p_4&=p_0(2q_4+q_2^2+2q_1q_3),\\
p_5&=2p_0(q_5+q_2q_3+q_1q_4),\\
p_6&=p_0(2q_6+2q_1q_5+2q_2q_4+q_3^2),
\endaligned
$$
with relations between $q_i$ such as:
$$0=q_7+2q_1q_6+2q_2q_5+2q_3q_4.$$
Conversely, $q_i$'s can be expressed through $p$'s:
$$
\aligned q_1&=\frac{p_1}{2p_0},\\
q_2&=\frac{1}{2p_0^2}\left(p_2p_0-\frac{p_1^2}{4}\right),\\
q_3&=\frac{1}{2p_0^3}\left(p_3p_0^2-\frac{p_1p_2p_0}{2}+\frac{p_1^3}{8}\right),\\
q_4&=\frac{1}{2p_0^4}\left\{p_4p_0^3-\frac{4p_2p_0-p_1^2}{8}-
\frac{p_1}{16}(8p_3p_0^2-4p_1p_2p_0+p_1^3)\right\},\\
q_5&=\frac{p_5}{2p_0}-q_2q_3-q_1q_4,\\
q_6&=\frac{1}{2p_0}(p_6-2q_1q_5p_0-2q_2q_4p-0-q_3^2p_0).
\endaligned
$$
The initial polynomial $X$ can be expressed through $q_i$'s:
$$
\frac{X}{p_0}=(1+q_1s+q_2s^2+q_3s^3)^2+2q_4s^4+2(q_1q_4+q_5)s^5+2(q_1q_5+q_2q_4+q_6)s^6.
$$

\subsection{The case of the Basic Algebraic Lemma}

\smallskip

We are going to determine polynomials $A, B, C$ of degree $\deg
A=3,\quad \deg B=\deg C=2$. Denote
$A=A(s)=A_0+A_1s+A_2s^2+A_3s^3$, $B=B(s)=B_0+B_1s+B_2s^2$,
$C=C(s)=C_0+C_1s+C_2s^2$. Then the equations (\ref{eq:halpsep})
and (\ref{eq:halpadd}) become
\begin{equation}\label{eq:halpsep2}
\aligned A-\sqrt{Y}-C(s-t)=0;\\
X-A\sqrt{Y}-AC(s-t)-Bs^{3}(s-t)=0;\\
 X-A^2=Bs^{3}(s-t).
 \endaligned
\end{equation}
We obtain  for $s=0$
$$
C_0=\frac{\sqrt{Y}-\sqrt{p_0}}{t}, \qquad A_0=\sqrt{p_0}.
$$
Then we calculate $A_i$, $i=1, 2$ from the last equation of
(\ref{eq:halpsep2}) by comparing polynomials $X$ and $A$ term by
term up to the second degree:
$$A_1=\frac{p_1}{2\sqrt{p_0}},\qquad
A_2=\frac{1}{2\sqrt{p_0}}\frac{4p_2p_0-p_1^2}{4p_0},$$ thus
$$A=\sqrt{p_0}(1+q_1s+q_2s^2+\lambda_1s^3).$$
From the relation $A(t)=\sqrt{Y}$ we get
$$
A_3=\frac{1}{t^3}\left[\sqrt{Y}-(\sqrt{p_0}+\frac{p_1}{2\sqrt{p_0}}t+
\frac{4p_0p_2-p_1^2}{8p_0^{3/2}}t^2)\right].
$$
The coefficients of $C$ are $C_1=A_2$ and $C_2=A_3$. The
coefficients of the polynomial $B$ are
$$
\aligned B_2&=p_6-A_3^2,\\
B_1&=B_2t+p_5-2A_2A_3,\\
B_0&=B_1t+p_4-(2A_1A_3+A_2^2).
\endaligned
$$
We factorize it
$$B=B_2(s-t_1^0)(s-t_1^1),$$
and denote
$$A(t_1^0)=-\sqrt{Y_1^0},\qquad A(t_1^1)=-\sqrt{Y_1^1}.$$
Now, we have
$$
\aligned
\frac{A+\sqrt{X}}{s-t_1^0}&=\frac{A+\sqrt{Y_1^0}}{s-t_1^0}+\frac{\sqrt{X}-\sqrt{Y_1^0}}
{x-y_1^0}\\
&=\frac{A(s)-A(t_1^0)}{s-t_1^0}+\frac{\sqrt{X}-\sqrt{Y_1^0}}
{x-y_1^0}\\
&=A_1+A_2(s+t_1^0)+A_3(s^2+st_1^0+t_1^{02})+\frac{\sqrt{X}-\sqrt{Y_1^0}}
{x-y_1^0}.
\endaligned
$$
Denote
$$
P_A^{(2)}(t,s):=A_1+A_2(s+t)+A_3(s^2+st+t^{2}).
$$
Then we have finally
$$Q_0=\frac{B_2(s-t_1^1)s^3}{P_A^{(2)}(t_1^0,s)+\frac{\sqrt{X}-\sqrt{Y_1^0}}
{x-y_1^0}}.$$

Step by step we get
\begin{equation}\label{eq:halpsep2i}
\aligned A^{(i)}=\sqrt{Y_i}-C^{(i)}(s-t_i);\\
 X-A^{(i)2}=B^{(i)}s^{3}(s-t_i),
 \endaligned
\end{equation}
where
$$\aligned
B^{(i-1)}(t_i)&=0,\quad t_i:=t_i^0,\\
A^{(i)}(t_{i+1})&=-\sqrt{Y_{i+1}}.
\endaligned
$$
Now, we have
\begin{equation}\label{eq:alpha2}
\aligned \beta_i&=B_2^{(i-1)}(s-t_i^1)s^3,\\
\alpha_i&=P_{A^{(i-1)}}^{(2)}(t_i, s)+ C^{(i)}.
\endaligned
\end{equation}
We can represent the HH continued fraction in the following manner
$$
\frac{\sqrt{X}-\sqrt{Y}}{x-y}=C+\frac{\beta_1|}{|\alpha_1}+\frac{\beta_2|}{|\alpha_2}+\dots
+\frac{\beta_i|}{|\alpha_i+Q_i},
$$
where
$$
Q_i=\frac{\sqrt{X}-\sqrt{Y_i}}{x-y_i}-C^{(i)}=\frac{B^{(i)}s^3}{\sqrt{X}+A^{(i)}}
$$
and
$$
Q_i=\frac{\beta_{i+1}}{\alpha_{i+1}+Q_{i+1}}.
$$

\medskip
\subsection{Relations between $\lambda_i$ and $t_i$}
\smallskip

From the equation (\ref{eq:halpsep2i}) we get
\begin{equation}\label{eq:halpsep22i}
\aligned
 X-A^{(i-1)2}&=B_2^{(i-1)}(s-t_i^1)s^{3}(s-t_{i-1})(s-t_i);\\
X-A^{(i)2}&=B_2^{(i)}(s-t_{i+1}^1)s^{3}(s-t_i)(s-t_{i+1});\\
 A^{(i)}(t_{i})&=\sqrt{Y_{i}};\\
 A^{(i-1)}(t_{i})&=-\sqrt{Y_{i}};\\
 A_3^{(i)}&=\sqrt{p_0}\lambda_i.
\endaligned
\end{equation}
From equations (\ref{eq:halpsep22i}) we get
\begin{equation}
\aligned
\lambda_i&=\frac{1}{t_i^3}\left(\frac{\sqrt{Y_i}}{\sqrt{p_0}}-(1+q_1t_i+q_2t_i^2)\right),\\
\lambda_{i-1}&=\frac{1}{t_i^3}\left(-\frac{\sqrt{Y_i}}{\sqrt{p_0}}-(1+q_1t_i+q_2t_i^2)\right),
\endaligned
\end{equation}
and thus
$$
t_i^3\sqrt{Y_{i+1}}+t_{i+1}^3\sqrt{Y_i}=\sqrt{p_0}(t_{i+1}-t_i)
[t_{1+1}^2+t_{i+1}t_i+q_1t_it_{i+1}(t_{i+1}+t_i)+q_2t_i^2t_{i+1}^2].
$$
From equations (\ref{eq:halpsep22i}) we also get
\begin{equation}\label{eq:vijet22}
\aligned
\lambda_{i-1}+\lambda_i&=-\frac{2}{t_i^3}(1+q_1t_i+q_2t_i^2),\\
\lambda_{i-1}\lambda_i&=-\frac{1}{t_i^6}\left[(1+q_1t_i+q_2t_i^2)^2-\frac{Y_i}{p_0}\right].
\endaligned
\end{equation}
Finally we have
\smallskip
\begin{proposition} If $\lambda_i$ is fixed, then $t_i,\quad t_{i+1},\quad t_{i+1}^1$
are roots of the polynomial $Q_X(\lambda_i,s)$ of degree 3 in $s$:
$$
\aligned
Q_X(\lambda_i,s):=(p_6-p_0\lambda_i^2)s^3+(p_5-2p_0q_2\lambda_i)s^2&+
(p_4-2p_0q_1\lambda_i-q_2^2p_0)s+\\
&+(p_3-2p_0\lambda_i-2p_0q_1q_2)=0.
\endaligned
$$
\end{proposition}

\medskip

\begin{corollary} Product of two consecutive $t_i$ and $t_{i+1}$ is
$$
t_it_{i+1}=\frac{p_3-2p_0\lambda_i-2p_0q_1q_2}{t_{i+1}^1(p_6-p_0\lambda_i^2)}.
$$
\end{corollary}
\medskip
The last Proposition can be reformulated giving relation
between two consecutive $\lambda_{i-1}$ and $\lambda_i$:
\medskip
\begin{proposition}\label{prop:lambda2}
If $t_i$ is fixed, then $\lambda_{i-1},
\lambda_i$ are solutions of quadratic equation:
$$
\lambda^2(-p_0t_i^3)+\lambda(-2p_0q_2t_i^2-2p_0q_1t_i-2p_0) +
(p_6t_i^3+p_5t_i^2+(p_4-q_2^2p_0)t_i+p_3-2p_0q_1q_2)=0.
$$
\end{proposition}
\medskip

\subsection{Normal form of genus 2 HH c. f. Recurrent relations}
\smallskip

Using equations (\ref{eq:vijet22}) and (\ref{eq:alpha2}) we get
formulae for $\alpha_i$:
$$
\alpha_i=\sqrt{p_0}\left(-\frac{2}{t_i}+(2q_2+\lambda_{i-1}t_i)s-
\frac{2}{t_i^3}(1+q_1t_i+q_2t_i^2)s^2\right).
$$
Given HH c. f. with $\alpha_i, \beta_i$, it can be transformed to
the equivalent with
$$\alpha_i'=c_i\alpha_i,\qquad \beta_i'=c_{i-1}c_i\beta_i.$$
Here we chose coefficients
$$c_i=-\frac{t_i}{2\sqrt{p_0}}$$
and get
\begin{equation}\label{eq:normal}
\aligned
\alpha_i'&=1+w_is+u_is^2,\\
\beta_i'&=v_i\frac{s-t_i^1}{t_i^1}s^3,
\endaligned
\end{equation}
where
\begin{equation}\label{eq:normal1}
\aligned u_i&=\frac{1+q_1t_i+q_2t_i^2}{t_i^2},\\
w_i&=-(q_2t_i+\frac{\lambda_{i-1}}{2}t_i^2),\\
v_i&=-\frac{\lambda_{i-1}}{2}+q_3.
\endaligned
\end{equation}
We are going to call {\it normal form} of the given HH c. f. the
form given with equation (\ref{eq:normal})

From the equations (\ref{eq:normal}) we get
\begin{equation}\label{eq:lambdai}
\lambda_{i-1}=-2v_i+q_3
\end{equation}
and
\begin{equation}\label{eq:ulambda}
\aligned (q_2-u_i)t_i^2+q_1t_i+1&=0,\\
\frac{\lambda_{i-1}}{2}t_i^2+q_2t_i+w_i&=0.
\endaligned
\end{equation}
From the equations (\ref{eq:ulambda}) we get
$$
t_i=\frac{(u_i-q_2)w_i+(v_i-\frac{q_3}{2})}{q_2(q_2-u_i)-q_1(\frac{q_3}{2}-v_i)}.
$$
From Proposition \ref{prop:lambda2} and equation
(\ref{eq:ulambda}) we get
$$
\lambda^2(-\frac{t_i}{2})-u_i\lambda+[q_6t_i+q_5(q_1t_i+1)+q_3u_i+q_4u_it_i]=0,
$$
having two zeroes $\lambda_{i-1}$ and $\lambda_i$. From the last
equation we get
$$
t_i=\frac{u_i(\lambda-q_3)-q_5}{-\frac{\lambda^2}{2}+q_6+q_5q_1+q_4u_i}.
$$
By using the second of equations (\ref{eq:ulambda}) and equating
the right sides of the last equation for $\lambda_{i-1}$ and
$\lambda$ we get
\medskip
\begin{lemma} The relation between $u_i$ and $v_i, v_{i+1}$
is
\begin{equation}
\aligned
-\frac{1}{2}(2u_iv_{i+1}+q_5)(q_3-2v_i)^2&+2u_i(q_6+q_5q_1+q_4u_i)(v_{i+1}-v_i)+\\
&+\frac{1}{2}(2u_iv_i+q_5)(q_3-2v_{i+1})^2=0.
\endaligned
\end{equation}
\end{lemma}

\medskip

From the last Lemma, we get

\medskip
\begin{corollary} If $v_i\neq v_{i+1}$ then
$$
0=q_3^2u_i-4u_iv_iv_{i+1}-2q_5(v_i+v_{i+1})+2q_5q_3-2u_i(q_6+q_5q_1+q_4u_i).
$$
\end{corollary}

\medskip
From the equations (\ref{eq:vijet22}), (\ref{eq:lambdai}),
(\ref{eq:normal1}) we get
\medskip
\begin{proposition} The recurrence equations connecting $v_i$ and
$v_{i+1}$ for fixed $u_i$ and $t_i$ are:

\begin{equation}
\aligned v_i+v_{i+1}&=\frac{u_i}{t_i}+q_3,\\
4v_iv_{i+1}&=(-2q_6-2q_1q_5-2q_4u_i)+q_3^2-2\frac{q_5}{t_i}.
\endaligned
\end{equation}
\end{proposition}
\medskip

We rewrite the polynomial $Q_X(\lambda_i,s)$ in the form
$$
Q_X(\lambda_i, s)=Q_3s^3+Q_2s^2+Q_1s+Q_0,
$$
where
$$
\aligned
Q_3&=q_6+q_1q_5+q_4q_2+\frac{q_3^2}{2}-\frac{\lambda_i}{2},\\
Q_2&=q_5+q_1q_4+q_2(q_3-\lambda_i),\\
Q_1&=q_4+q_1(q_3-\lambda_i),\\
Q_0&=q_3-\lambda_i.
\endaligned
$$
Summing the relations $Q_X(\lambda_i, t_i)=0$ and $Q_X(\lambda_i,
t_{i+1})=0$ we get
$$
(u_i+u_{i+1})(q_3-\lambda_i)=(t_i+t_{i+1})\left(\frac{\lambda_i^2}{2}-q_6-q_5q_1-q_4q_2\right)-
q_4\frac{t_i+t_{i+1}}{t_it_{i+1}}-2q_5-2q_1q_4.
$$
From the last equation by using the Viete formulae for the
polynomial $Q_X(s)$ and equation (\ref{eq:lambdai}) we get
\medskip
\begin{proposition}
\begin{equation}
\aligned u_i+u_{i+1}=
\frac{1}{2}v_{i+1}[&\left(\frac{Q_2}{Q_3}-t_{i+1}^1\right)\left(\frac{(-2v_{i+1}+q_3)^2}{2}
-q_6-q_5q_1-q_4q_2\right)-\\
&-q_4\left(\frac{Q_2}{Q_3}-t_{i+1}^1\right)\frac{Q_3}{Q_0}t_{i+1}^1-2q_5-2q_1q_4].
\endaligned
\end{equation}
\end{proposition}
\medskip

\section{Periodicity and symmetry}\label{sec:sim}

\smallskip
\subsection{Definition and the first properties}
\medskip
According to the Theorem \ref{th:rel2} in the case
$$t_h=t_k$$
for some $h, k$ there are two possibilities:
$$
\aligned (I) \quad \lambda_{h-1}&=\lambda_{k-1},\qquad \lambda_h=\lambda_k;\\
(II)\quad \lambda_{h-1}&=\lambda_{k},\qquad
\lambda_h=\lambda_{k-1}.
\endaligned
$$
The first possibility leads to {\it periodicity}:
$$
t_{h+s}=t_{k+s},\qquad \lambda_{h+s}=\lambda_{k+s}
$$
for any $s$ and with appropriate choice of roots. If $p=h-k$ and
$r\cong s (mod p)$ then
$$
\alpha_r=\alpha_s, \qquad \beta_r=\beta_s.
$$

The second possibility leads to {\it symmetry}:
$$
t_{h+s}=t_{k-s},\qquad \lambda_{h+s}=\lambda_{k-s-1}
$$
for any $s$. More precisely, we introduce

\begin{definition}

\begin{itemize}
\item{(i)} If $h+k=2n$ we say that HH c. f. is {\bf even
symmetric} with
$$
\alpha_{n-i}=\alpha_{n+i}, \qquad \beta_{n-i}=\beta_{n+i-1}.
$$
for any $i$ and with $\alpha_n$ as the {\bf centre of symmetry}.
\item{(ii)} If $h+k=2n+1$ we say that HH c. f. is {\bf odd
symmetric} with
$$
\alpha_{n-i}=\alpha_{n+i-1}, \qquad \beta_{n-i}=\beta_{n+i}.
$$
for any $i$ and with $\beta_n$ as the {\bf centre of symmetry}.
\end{itemize}
\end{definition}

\smallskip

Now we can formulate some initial properties connecting
periodicity and symmetry.

\begin{proposition}\label{prop:persym1}
\begin{itemize}
\item{(A)}If a HH c. f. is periodic with the period of $2r$ and
even symmetric with $\alpha_n$ as the centre, then it is also even
symmetric with respect $\alpha_{n+r}$. \item{(B)}If a HH c. f. is
periodic with the period of $2r$ and odd symmetric with respect
$\beta_n$, then it is also odd symmetric with respect
$\beta_{n+r}$. \item{(C)}If a HH c. f. is periodic with the period
of $2r-1$ and even symmetric with respect $\alpha_n$, then it is
also odd symmetric with respect $\beta_{n+r}$. The converse is
also true.
\end{itemize}
\end{proposition}

\smallskip
\begin{proposition}\label{prop:persym2} If a HH c. f. is double
symmetric, then it is periodic. Moreover:
\begin{itemize}
\item{(A)}If a HH c. f. is even symmetric with respect $\alpha_m$
and $\alpha_n$, $n<m$ then the period is $2(n-m)$. \item{(B)}If a
HH c. f. is odd symmetric with respect $\beta_m$ and $\beta_n$,
$n<m$ then the period is $2(m-n)$.\item{(C)}If a HH c. f. is even
symmetric with respect $\alpha_n$ and $\beta_m$,  then the period
is $2(n-m)+1$ in the case $m\le n$ and the period is $2(m-n)-1$
when $m>n$.
\end{itemize}
\end{proposition}

\smallskip

{\bf Observations:}
\begin{itemize}
\item{(i)} A HH c. f. can be at the same time even symmetric and
odd symmetric. \item{(ii)}If  $\lambda_i=\lambda_{i-1}$ then the
symmetry is even; if $t_i=t_{i+1}$ then the symmetry is odd.
\end{itemize}
\medskip
\subsection{Further results}
\medskip
\begin{theorem} An H. H. c. f. is even-symmetric with the central
parameter $y$ if $X(y)=0$.
\end{theorem}
\medskip
The proof follows from the fact that even-symmetry is equivalent
to the condition $\lambda_p=\lambda_{p-1}$, which is equivalent to
the condition $Y_p=0$.
\medskip
For odd-symmetry let us start with the example of genus two case.
From the relations
\begin{equation}\label{eq:oddsymm}
Q_X(\lambda, s)=0,\qquad  \frac{d}{ds}Q_X(\lambda, s)=0
\end{equation}
we get the system
\begin{equation}\label{eq:oddsymm2}
\aligned 3Q_3s^2+2Q_2s+Q_1&=0\\
Q_2s^2+2Q_1s+3Q_0&=0.
\endaligned
\end{equation}
From the last system we get
$$
v_{i+1}=-\frac{s[q_5+q_1q_4)s+q_4]}{2q_2s^2+4q_1s+6}
$$
or, equivalently
$$
\lambda_i=\frac{p_5s^2+2(p_4-q_2^2p_0)s+3(p_3-2p_0q_1q_2)}{2p_0q_2s^2+4p_0q_2s+6p_0}.
$$

By replacing any of the last two relations in the first equation
of the relations (\ref {eq:oddsymm2}) we get the equation of the
sixth degree in $s$. On the other hand, from the relations (\ref
{eq:oddsymm2}) we get
$$
s=\frac{9Q_0Q_3-Q_1Q_2}{2Q_2^2-6Q_1Q_3}.
$$
Now, by replacing the last formula in the first  equation of the
relations (\ref {eq:oddsymm2}) we get the equation of the eight
degree in $\lambda_i$.
\medskip

\section{General case}

\subsection{Invariant approach}

\medskip

Now we pass to the general case, with a polynomial $X$ of degree
$2g+2$. Relation
$$
Q_X(\lambda, s)=0
$$
defines a {\it basic curve} $\Gamma_X$. Denote its genus as $G$
and consider its projections $p_1$ to the $\lambda$-plane, and
$p_2$ to the $s$-plane. The ramification points of the second
projection we denote $R_e$ and call them {\it even-symmetric
points} of the basic curve. The ramification points of the first
projection, denoted as $R_{o+r}$, are  union of {\it the
odd-symmetric points} and of {\it the gluing points}. The gluing
points represent situation where some of the roots of the
polynomial $B^{(i)}$ coincide. For example in genus 2 case the
gluing points correspond to the condition $t_{i+1}=t_{i+1}'$.
\medskip
From the last theorem we get
$$
\deg R_e=2g+2.
$$
By applying the Riemann-Hurvitz formula we have
$$\aligned
2-2G&=4-\deg R_e;\\
2-2g&=2(g+1)-\deg R_{o+r}.
\endaligned
$$
Thus
$$ genus(\Gamma_X)=G=g
$$
and
$$
\deg R_{o+r}=4g.
$$
We get a birational morphism
$$
f:\Gamma \to \Gamma_X
$$
by the formulae
$$
f:(x,s)\mapsto(t,\lambda),
$$
where
$$\aligned
t&=x,\\
 \lambda &=\frac{1}{t^{g+1}}\left(\frac {s}{\sqrt{p_0}}-Q_g(t)\right),\\
 Q_g(t)&=1+q_1t+\dots+q_gt^g.
 \endaligned
 $$
 The function $f$ satisfies commuting relation
 $$
 f\circ\tau_{\Gamma}=\tau_{\Gamma_X}\circ f,
 $$
where $\tau_{\Gamma}$ and $\tau_{\Gamma_X}$ are naturale
involutions on the hyperelliptic curves $\Gamma$ and $\Gamma_X$
respectively.
\medskip

\subsection{Multi-valued divisor dynamics}

\medskip

The inverse image of  a value $z$ of the function $\lambda$ is a
divisor of degree $g+1$:
$$
\lambda^{-1}(z)=:D(z), \qquad \deg D(z)=g+1.
$$
Now, the HH-continued fractions development can be described as a
multi-valued discrete dynamics of divisors $D_k^j=D(z_k^j)$. Here
the lower index $k$ denotes the $k$-th step of the dynamics and
the upper index $j$ goes in the range from $1$ to $(g+1)k$
denoting branches of multivaludness. More precisely, the discrete
divisor dynamics which governs HH-continued fraction development
can be described as follows. Suppose the development has started
with a point $P_0=P_0^1$. It leads to the divisor
$$D_0:=D(\lambda(P_0))=P_0^1+P_0^2+\dots +P_0^{g+1},$$
with $\lambda(P_0^i)=\lambda(P_0^j).$ In the next step we get
$g+1$ divisors of degree $g+1$:
$$D_1^j:=D\left(\lambda(\tau_{\Gamma}(P_0^j))\right).$$
And we continue like this. In each step, a divisor from the
previous step
$$D_{k-1}^j=P_{k-1}^{(j, 1)}+\dots +P_{k-1}^{(j,
(g+1))}
$$
gives $g+1$ new divisors
$$
D_k^{(j-1)(g+1)+l}:=D\left(\lambda(\tau_{\Gamma}(P_{k-1}^{(j,
l)}))\right),\qquad l=1,\dots,g+1.
$$
In the case of genus one, this dynamics can be traced out from the
$2-2$ - correspondence $Q_{\Gamma}(\lambda, t)=0$. According to
\cite{Ha}, for example, there exist constants $a, b, c, d, T$ such
that for every $i$ we have
$$
\lambda_i=\frac{ax(u_i+T)+b}{cx(u_i+T)+d},
$$
where $u$ is an uniformizing parameter on the elliptic curve. The
involution is symmetry at the origin and since the function $x$ is
even, the two parameters corresponding to the fixed value
$\lambda_i$ are $u_i$ and $\bar u_i=-u_i-2T$. Thus
$$
\aligned
u_{i+1}&=u_i+2T,\\
\lambda_{i+1}&=\frac{ax(u_i+3T)+b}{cx(u_i+3T)+d}.
\endaligned
$$
\medskip

In the cases of higher genera the dynamics is much more
complicated. Thus we pass to the consideration of generalized
Jacobians.
\medskip

\subsection{Generalized Jacobians}
\medskip

A natural environment for consideration divisors of degree $g+1$
on the curve $\Gamma$ of genus $g$ is generalized Jacobian of
$\Gamma$ obtained by gluing a pair of points $Q_1, Q_2$ of
$\Gamma$, denoted as $\Jac (\Gamma, \{Q_1, Q_2\})$ (see
\cite{Fa}).

It can be understood as a set of classes of relative equivalence
among the divisors on $\Gamma $ of certain degree. Two divisors of
the same degree  $D_1$ and $D_2$ are called {\it equivalent
relative to the points} $Q_1, Q_2$, if there exists a function $f$
meromorphic on $\Gamma $ such that $(f)=D_1-D_2$ and
$f(Q_1)=f(Q_2)$.

The generalized Abel map is defined with
$$
\tilde A(P)=(A(P),\mu_1 (P),\mu_2(P)),\quad
\mu_i(P)=exp\int_{P_0}^P\Omega_{Q_iQ_0}, i=1,2,
$$
and $A(P)$ is the standard Abel map. Here $\Omega_{Q_iQ_0}$
denotes the normalized differential of the third kind, with poles
at the point $Q_i$ and at arbitrary fixed point $Q_0$.

Here we consider the case where $Q_1=+\infty$ and $Q_2=-\infty$ on
the curve $\Gamma$ of genus $g$. The divisors we are going to
consider are those of degree $g+1$ of the form $D_i=D(z_i)$ where
usually $z_i=\lambda(P_i)$. The divisors of degree $g+1$ up to the
equivalence relative to the points $Q_1$ and $Q_2$ are uniquely
determined by their generalized Abel image on the generalized
Jacobian.

Thus, in order to measure the distance between relative classes of
$D_1=D(z_1)=D(\lambda(P_1))$ and of $D_2=D(z_2)=D(\lambda(P_2))$
we introduce the following index
$$
I(D_1, D_2)=I(z_1, z_2)=I(P_1, P_2):=\frac{\lim_{P\to
+\infty}\frac{\lambda(P)-z_1}{\lambda(P)-z_2}}{\lim_{P\to
-\infty}\frac{\lambda(P)-z_1}{\lambda(P)-z_2}}.
$$
We are interested in the case $P_2=\tau_{\Gamma}(P_1)$ and we have
$$
I(P_1):=I(P_1,\tau_{\Gamma}(P_1))=\lim_{P\to
+\infty}\frac{\frac{\lambda(P)-\lambda(P_1)}{\lambda(P)-\lambda(\tau(P_1))}}{
\frac{\lambda(\tau(P))-\lambda(P_1)}{\lambda(\tau(P))-\lambda(\tau(P_1))}}
$$
After some calculations we get
\medskip
\begin{lemma} The index of the point is given by the formula
$$I(P_1)=1+\frac{2\sqrt{p_{2g+2}}(\lambda(\tau(P_1))-\lambda(P_1))}{p_{2g+2}-
\sqrt{p_{2g+2}}(\lambda(\tau(P_1))-\lambda(P_1))-\lambda(P_1)\lambda(\tau(P_1))}.
$$
\end{lemma}
\medskip

\section{Irregular terms}\label{sec:irreg}
\medskip If some of the parameters $t$ appear to be infinite or
zero, we are going to call them {\it irregular}.

\subsection{$t_h$ - infinite}
\medskip Suppose $t_0=\infty$. We start from the following
$$
X-A^2=Bs^{g+1}.
$$
Then, HH continued fraction is based on the relation
$$
\sqrt{X}-\sqrt{p_{2g+2}}s^{g+1}=C +\frac{Bs^{g+1}}{\sqrt{X}+A}.
$$
\medskip
\begin{proposition} Irregular HH c. f. with $t_h=\infty$ is even
symmetric if and only if $p_{2g+2}=0$.
\end{proposition}
\medskip

\subsection{$t_h=0$}
\smallskip
Let $t_0=0$. In that case the basic relation of HH continued
fraction is
$$
\frac{\sqrt{X}-\sqrt{p_0}}{x-\epsilon}-C=\frac{B(x-\epsilon)^{g+1}}{\sqrt{X}+A}.
$$
Then we have also
$$
\aligned
A-\sqrt{p_o}&=Cs,\\
X-A^2&=Bs^{g+2}.
\endaligned
$$
An HH continued fraction is developed through the following
relations
$$
\aligned
\sqrt{X}&=A+\frac{Bs^{g+2}}{\sqrt{X}+A},\\
\sqrt{X}&=A+\frac{Bs^{g+2}}{P_A^{(g)}}+\frac{\sqrt{X}-\sqrt{Y_1}}{x-y_1}
\endaligned
$$
\medskip
\begin{proposition} The condition $t_h=0$ is equivalent to
$v_{h+1}=\infty$. Such an HH c. f. is odd symmetric with respect
to $\beta_{h+1}$.
\end{proposition}
\medskip

\subsection{$\epsilon$-infinite}
\smallskip
The starting relation in the case $\epsilon=\infty$ is
$$
X-A^2=B(x-y).
$$
Changing the variables $x=1/s, y=1/t$ we come to
$$
X'-A'^2=-\frac{1}{t}B's^{g+1}(s-t).
$$
The HH c. f. takes the form
$$
\frac{\sqrt{X}-\sqrt{Y}}{x-y}=C+\frac{B_0|}{|A_1}+\frac{B_0^{(1)}|}{|A_2}+\dots+
\frac{B_0^{(i-1)}|}{|A_i+Q_i},
$$
where $\deg B_0^{(i)}=g-1$, $B_0^{(i)}=B^{(i)}/(x-t_i^0)$, $\deg
C=g$, $\deg A_i=g$. Appropriate HH c. f. is obtained from the last
one after the change of variables.
\medskip
\begin{lemma} The identity holds
$$
y^{g+1}\frac{\sqrt{X}-\sqrt{Y}}{x-y}=(x^g+x^{g-1}y+\dots+xy^{g-1}+y^g)\sqrt{Y}+
\frac{y^{g+1}\sqrt{X}-x^{g+1}\sqrt{Y}}{x-y}.
$$
\end{lemma}
\medskip

\begin{proposition}The HH element $(\sqrt{X}-\sqrt{Y})/(x-y)$
around $x=\infty$ has the same coefficient as
$(\sqrt{X'}-\sqrt{Y'})/(s-t)$ around $s=0$.
\end{proposition}

\medskip

\section{Remainders, continuants and approximation}

\medskip

We consider an HH c. f. of an element $f$
$$
f=C+\frac{\beta_1|}{|\alpha_1}+\frac{\beta_2|}{|\alpha_2}+\dots .
$$
Together with {\it the remainder of rank $i$} $Q_i$, where
$$
Q_i=\frac{B^{(i)s^{g+1}}}{\sqrt{X}+A^{(i)}},
$$
we consider {\it the continuants} $(G_i)$ and $(H_i)$ and {\it the
convergents} $G_i/H_i$ such that
$$
\left[\begin{array}{llll}
G_{m} & G_{m-1}\\
H_{m} & H_{m-1}\\
\end{array}\right]=T_CT_1\cdots T_m.
$$
Here
$$
T_i=\left[\begin{array}{llll}
\alpha_{i} & 1\\
\beta_i & 0\\
\end{array}\right]
$$
and
$$
T_C=\left[\begin{array}{llll}
C & 1\\
1 & 0\\
\end{array}\right].
$$
By taking the determinant of the above matrix relation we get
\begin{equation}\label{eq:det}
\aligned G_mH_{m-1}-G_{m-1}H_m&=(-1)^{m-1}\beta_1\beta_2\dots
\beta_m\\
&=\delta_ms^{(g+1)m}\\
\deg\delta_m&=(g-1)m.
\endaligned
\end{equation}
We also have the following relations
$$
f=\frac{(\alpha_m+Q_m)G_{m-1}+\beta_mG_{m-2}}{(\alpha_m+Q_m)H_{m-1}+\beta_mH_{m-2}}
=\frac{G_m+Q_mG_{m-1}}{H_m+Q_mH_{m-1}},
$$
and
$$
Q_m=-\frac{G_m-H_mf}{G_{m-1}-H_{m-1}}.
$$
\medskip
\begin{proposition}\label{prop:deg} The degree of the continuants
is $\deg G_m=g(m+1)$, $\deg H_m=gm$.
\end{proposition}
\medskip
Let us introduce
$$
\aligned \hat G_m&=G_m+\frac{H_m}{s-t}\sqrt{Y}\\
\hat H_m&=\frac{H_m}{s-t}.
\endaligned
$$
Then we have
$$
Q_m=-\frac{\hat G_m-\hat H_m\sqrt{X}}{\hat G_{m-1}-\hat
H_{m-1}\sqrt{X}}
$$
and also
\begin{equation}\label{eq:hat}
\aligned \hat G_mA^{(m)}+\hat G_{m-1}B^{(m)}s^{g+1}&=\hat H_mX\\
\hat H_mA^{(m)}+\hat H_{m-1}B^{(m)}s^{g+1}&=\hat G_mX.
\endaligned
\end{equation}
From the last equations we get
$$
\aligned
\delta_ms^{(g+1)m}A^{(m)}&=P_1(s)\\
\delta_ms^{(g+1)(m+1)}B^{(m)}&=P_2(s),
\endaligned
$$
with
$$
\aligned
P_1(s):=H_mH_{m-1}\frac{X_Y}{x-y}-(G_mH_{m-1}-G_{m-1}H_m)\sqrt{Y}-G_mG_{m-1}(s-t)\\
P_2(s):=G_m^2(s-t)+2G_mH_m\sqrt{Y}-H_m^2\frac{X-Y}{x-y}.
\endaligned
$$
\medskip
\begin{theorem}\label{th:aprox1}
\begin{itemize}
\item{(A)} The polynomial $G_mH_{m-1}-H_mG_{m-1}$ is of degree
$2gm$. The first $(g+1)m$ coefficients are zero. \item{(B)} The
polynomial $P_1$ is of degree $2mg+g+1$. Its first $(g+1)m$
coefficients are zero. \item {(C)} The polynomial $P_2$ is of
degree $2mg+2g+1$ and its $(g+1)(m+1)$ first coefficients are
zero.
\end{itemize}
\end{theorem}
\medskip
\begin{lemma}\label{lemma:lm} The following relations hold
$$
\aligned
\frac{G_{m-1}(t_m)}{H_{m-1}(t_m)}&=-A^{(m)}(t_m)=A^{(m-1)}(t_m)\\
\hat G_m-\hat H_m\sqrt {X}&=(-1)^{m+1}Q_0Q_1Q_2\dots Q_m.
\endaligned
$$
\end{lemma}
\medskip
\begin{theorem}\label{th:aprox2} If $X(\epsilon)\neq 0$ and
$\epsilon\neq y$, then the element
$$
\hat G_m-\hat H_m\sqrt{X}=G_m-H_m\frac{\sqrt{X}-\sqrt{Y}}{x-y}
$$
has a zero of order $(g+1)(m+1)$ at $s=0$. If $H(0)\neq 0$ then
the differences
$$
\frac{\sqrt{X}-\sqrt{Y}}{x-y}-\frac {G_m}{H_m},\qquad
\sqrt{X}-\frac{\hat G_m}{\hat H_m}
$$
have developments starting with the order of $s^{(g+1)(m+1)}$.
\end{theorem}
\medskip

Now, we consider $\sqrt {X}$ and its development as HH c. f. In
that case, starting from
$$
\frac{\sqrt{X}-\sqrt{p_0}}{x-\epsilon},
$$
we have
$$
\deg G_0=g+1,\quad H_0=1,\quad H_1=\alpha_1,\quad
G_1=\alpha_1G_0+\beta_1s^{g+2}
$$
and
$$
\aligned G_m&=\alpha_mG_{m-1}+\beta_mG_{m-2},\\
H_m&=\alpha_mH_{m-1}+\beta_mH_{m-2}.
\endaligned
$$
From the last relation we have
\medskip
\begin{theorem}\label{th:aprox3}
\begin{itemize}
\item{(A)} The degree of the continuants in this case is $\deg
G_m=g(m+1)+1$, $\deg H_m=gm$. \item {(B)} If $y=\epsilon$ then the
development of the difference
$$
\sqrt{X}-\frac{\hat G_m}{\hat H_m}
$$
starts with the order $s^{(g+1)(m+1)+1}$.
\end{itemize}
\end{theorem}
\medskip
\begin{theorem}\label{th:aprox4}
\begin{itemize}
\item{(A)} The polynomial $G_mH_{m-1}-H_mG_{m-1}$ is of degree
$2gm+1$ in $s$. The first $(g+1)m+1$ coefficients are $0$.
\item{(B)} The polynomial $H_mH_{m-1}X-G_mG_{m-1}$ is of degree
$2mg+g+2$. Its first $(g+1)m+1$ coefficients are zero. \item {(C)}
The polynomial $G^2_m-H^2_mX$ is of degree $2mg+2g+2$ and its
$(g+1)(m+1)+1$ coefficients are zero.
\end{itemize}
\end{theorem}
\medskip
There are infinite ways to calculate $\sqrt{X}$ in the
neighborhood of $\epsilon$, depending on choice of the parameter
$y$. The best approximation one obtains for the choice
$$y=\epsilon.$$
To conclude the last observation we need to check the case
$y=\infty$. In this case we have:
$$
G_0=\sqrt{p_0}(1+q_1s+\dots+q_gs^g),\quad H_0=1,\quad
G_1=\alpha_1G_0+\beta_1,\quad H_1=\alpha_1,
$$
and we denote
$$
\hat G_m=G_m+\sqrt{a_0}H_ms^{g+1},\quad \hat H_m=H_m.
$$
Then we have
\medskip
\begin{proposition}\label{prop:infty}
\begin{itemize}
\item{(A)} The degree of continuants is $\deg G_m=g(m+1)$, $\deg
H_m=gm$. \item{(B)} If $X(\epsilon)\neq 0$ and suppose the
parameters $t_1,\dots, t_{m+1}$ are finite and different from
zero, then
$$
\sqrt{X}-\frac{\hat G_m}{\hat H_m}
$$
has the development starting with order $2m+2$ in $s$.
\end{itemize}
\end{proposition}

\medskip

\section{Conclusion: Polynomial growth and integrability}

\medskip

Due to the well-known facts, the Pade approximants of a
hyperelliptic functions are unique up to the scalar factors. The
approximants discussed in the previous section in the case of
genus higher than 1 are neither unique nor of Pade type. By
construction they have an exponential growth on the first sight.
But careful analysis of their degrees comparing to their degrees
of approximation done in the previous section indicates their real
polynomial growth. After Veselov, we can consider a discrete
multi-valued dynamics to be integrable one if it has polynomial
growth instead of an exponential one. Thus, in that sense, we can
say that the multi-valued discrete dynamics associated with
HH-continued fractions is an integrable dynamics.

In the case of genus one, it can be seen as multi-valued discrete
dynamics associated with the Euler-Chasles 2-2 correspondence,
which has been studied by Veselov (see \cite {Ve}) and Veselov and
Bucshtaber (see \cite {BV}). It would be quite interesting to
consider higher genus dynamics from the point of view of
$n$-valued groups and their actions, following Buchstaber (see
\cite {Bu}).

\medskip
\subsection*{Acknowledgements}

The research was partially supported by the Serbian Ministry of
Science and Technology, Project {\it Geometry and Topology of
Manifolds and Integrable Dynamical Systems}. The author would like
to thank Professor Marcel Berger for indicating the Halphen book
\cite{Ha}.

\end{document}